\documentclass[12pt]{amsart}

\usepackage{amsmath}
\usepackage{amssymb}
\usepackage{amsfonts}
\usepackage{amsthm}
\usepackage{enumerate}
\usepackage{color}
\usepackage{psfrag}
\usepackage{mathrsfs}
\usepackage{amscd}
\usepackage[hyperindex,backref=page]{hyperref}
\usepackage{graphicx,color}
\usepackage{tikz} 

\newtheorem{thm}{Theorem}[section]
\newtheorem{cor}[thm]{Corollary}
\newtheorem{lemma}[thm]{Lemma}

\theoremstyle{definition}
\newtheorem{defi}[thm]{Definition}
\newtheorem{ex}[thm]{Example}
\newtheorem{nota}[thm]{Notation}
\newtheorem{setup}[thm]{Setup}
\newtheorem{const}[thm]{Construction}

\theoremstyle{remark}
\newtheorem{rem}[thm]{Remark}

\begin{document}

\title[the cover ideals of clique-whiskered graphs]{On minimal free resolutions of the cover ideals of clique-whiskered graphs}

\author{Yuji Muta and Naoki Terai}

\address[Yuji Muta]{Department of Mathematics, Okayama University, 3-1-1 Tsushima-naka, Kita-ku, Okayama 700-8530 Japan.}

\email{p8w80ole@s.okayama-u.ac.jp}

\address[Naoki Terai]{Department of Mathematics, Okayama University, 3-1-1 Tsushima-naka, Kita-ku, Okayama 700-8530, Japan.}

\email{terai@okayama-u.ac.jp}


\begin{abstract}
We explicitly construct a minimal free resolution of the cover ideals of clique-whiskered graphs, which were introduced in \cite[Section 3]{cn}. In particular, Cohen--Macaulay chordal graphs, clique corona graphs, and Cohen--Macaulay Cameron--Walker graphs are examples of clique-whiskered graphs. We also introduce multi-clique-whiskered graphs as a generalization of both clique-whiskered graphs and multi-whisker graphs. We prove that multi-clique-whiskered graphs are vertex decomposable and hence sequentially Cohen--Macaulay. Moreover, we provide formulas for the projective dimension and the Castelnuovo--Mumford regularity of their edge ideals. Finally, we construct minimal free resolutions of the cover ideals of both multi-clique-whiskered graphs and very well-covered graphs.
\end{abstract}






\maketitle


\section{Introduction}
Studying the minimal free resolutions of monomial ideals in polynomial rings plays an important role in combinatorial commutative algebra. These resolutions contain a wealth of information and provide several invariants in homological algebra. However, explicitly constructing a minimal free resolution remains a particularly challenging problem. Many researchers have been constructing minimal free resolutions of monomial ideals \cite{ahh, bps, bs, c, cf, j, ek, hh, h1, kty, mm, y}.

Throughout, we assume that graphs mean finite simple graphs, namely, finite undirected graphs without loops or multiple edges and let $\Bbbk$ be an arbitrary field. Let $G=(V(G), E(G))$ be a graph, where $V(G)$ denotes the vertex set of $G$, $E(G)$ denotes the edge set of $G$, and $\Bbbk[V(G)]=\Bbbk[x_1,\ldots, x_{n}]$ be a polynomial ring over $\Bbbk$. Then the {\it edge ideal}, denoted by $I(G)$, is the ideal  of $\Bbbk[V(G)]$ defined by$$I(G)=(x_{i}x_{j}\,\,: \{x_{i}, x_{j}\}\in E(G)).$$This ideal was first introduced by Villarreal in his 1990 paper \cite{v1}. Also, the {\it cover ideal} of $G$, denoted by $J(G)$, is the ideal defined by $$J(G)=(x_{i_{1}}\cdots x_{i_{s}}\,\,:\{i_{1},\ldots, i_{s}\}\mbox{ is a vertex cover of }G).$$Cover ideals correspond to the Alexander dual ideals of edge ideals. 

In this paper, we treat clique-whiskered graphs, which were first introduced by Cook and Nagel in 2012 in the context of combinatorial commutative algebra \cite[Section 3]{cn}. One can see that Cohen--Macaulay chordal graphs, clique corona graphs and Cohen--Macaulay Cameron--Walker graphs are clique-whiskered graphs (see Lemmas \ref{chordal}, \ref{corona}, \ref{CW}).  

For cover ideals, Herzog and Hibi provided an explicit minimal free resolution of Cohen--Macaulay bipartite graphs \cite[Theorem 2.1]{hh}. Later, Mohammadi and Moradi studied a minimal free resolution of unmixed bipartite graphs \cite{mm}. Moreover, Kimura, Terai and Yassemi successfully constructed an explicit minimal free resolution of Cohen--Macaulay very well-covered graphs \cite[Theorem 3.2]{kty}. This result was reproved by Crupi and Ficarra using a technique known as Betti splittings \cite[Theorem 3.5]{cf}. In particular, it is known that whisker graphs are Cohen--Macaulay very well-covered graphs, and thus, an explicit minimal free resolution of whisker graphs can also be determined. In section \ref{resol of clique}, we construct explicitly a minimal free resolution of cover ideals of clique-whiskered graphs (see Theorem \ref{min}). This result leads us to give the graded Betti numbers of the cover ideals of clique-whiskered graphs, and hence we get the graded Betti numbers of the cover ideals of Cohen--Macaulay chordal graphs, clique corona graphs and Cohen--Macaulay Cameron--Walker graphs respectively. Also, we obtain formulas for the Hilbert series of the local cohomology modules of residue rings modulo edge ideals associated with these graphs. 

In section \ref{multi}, we define multi-clique-whiskered graphs as a generalization of clique-whiskered graphs. It is also a generalization of multi-whisker graphs, which were first introduced by Pournaki and the authors of this paper \cite[Section 1]{mpt} in the context of combinatorial commutative algebra. We show that multi-clique-whiskered graphs are vertex decomposable and hence sequentially Cohen--Macaulay, which is a generalization of \cite[Theorem 3.3]{cn} and \cite[Theorem 6.1]{mpt} (see Theorem \ref{seq CM}). Also, we show that the regularity of edge ideals of multi-clique-whiskered graphs is equal to the its induced matching number (see Corollary \ref{reg of multi}).  Moreover, we construct explicit minimal free resolutions of cover ideals of multi-clique-whiskered graphs and very well-covered graphs (see Theorems \ref{min resol2}, \ref{min resol3}). The cover ideals of these graphs do not necessarily admit linear resolutions. As far as we know, an explicit construction of the minimal free resolution has not yet been given in such cases. This result leads us to a formula for the Hilbert series of the local cohomology modules of residue rings modulo edge ideals of multi-whisker graphs \cite[Theorem 4.1]{mpt}, as well as to a different form of the Hilbert series of the local cohomology modules of residue rings modulo edge ideals of very well-covered graphs, as presented in \cite[Theorem 4.4]{kpty}.


\section{Preliminaries}
In this section, we recall several definitions and known results from graph theory and combinatorial commutative algebra that will be used throughout this paper. We refer the reader to \cite{bh, s1, v0, v2} for the detailed information on combinatorial and algebraic background. 

\subsection{Combinatorial background.}
Let $G=(V(G), E(G))$ be a graph, where $V(G)$ denotes the vertex set and $E(G)$ denotes the edge set of $G$. For a subset $W$ of $V(G)$, we denote the induced subgraph on $W$ by $G|_{W}$. A subset $C$ of $V(G)$ is called {\it vertex cover}, if $C\cap e\neq\emptyset$ for every edge $e$ of $G$. In particular, if $C$ is a vertex cover that is minimal with respect to inclusion, then $C$ is called a {\it minimal vertex cover} of $G$. Let ${\rm Min}(G)$ be the set of minimal vertex covers of $G$. A subset $A$ of $V(G)$ is called an {\it independent set}, if no two vertices in $A$ are adjacent to each other. In particular, if $A$ is an independent set of $G$ that is maximal with respect to inclusion, then $A$ is called a {\it maximal independent set} of $G$. Next, let us recall the definition of Stanley--Reisner rings. Let $n$ be a positive integer and let $[n]=\{1,\ldots, n\}$. A {\it simplicial complex }$\Delta$ on $[n]$ is a nonempty subset of the power set $2^{[n]}$ of $[n]$ such that $G\in\Delta$, if $G\subset F$ and $F\in \Delta$. An element of $\Delta$ is called face. For a subset $W$ of $[n]$, we set $$\Delta_{W}=\{F\in\Delta\,\,: F\subset W\}.$$Also, the {\it Alexander dual complex} of $\Delta$, denoted by $\Delta^{\vee}$, defined as$$\Delta^{\vee}=\{F\in 2^{[n]}\,\,: [n]\setminus F\notin\Delta\}.$$The {\it Stanley--Reisner ideal} $I_{\Delta}$ of $\Delta$ is the ideal generated by squarefree monomials corresponding to minimal non-faces of $\Delta$, namely,$$I_{\Delta}=(x_{i_{1}}\cdots x_{i_{r}}\,\,: \{i_{1},\ldots, i_{r}\}\mbox{ is a minimal non-face of }\Delta).$$
Also, $\Bbbk[\Delta]=\Bbbk[x_{1},\ldots, x_{n}]/I_{\Delta}$ is called the {\it Stanley--Reisner ring} of $\Delta$. For simplicial complexes $\Delta$ and $\Delta^{\prime}$ whose vertex sets are disjoint, the {\it join} $\Delta*\Delta^{\prime}$ is the simplicial complex whose faces $F\cup F^{\prime}$, where $F\in\Delta$ and $F^{\prime}\in\Delta^{\prime}$. Let $\Delta$ be a $(d-1)$-dimensional simplicial complex and let $f_{i}$ be the number of faces of $\Delta$ of dimension $i$. Then the sequence $f(\Delta)=(f_{0}, f_{1},\ldots, f_{d})$ is called the $f$-{\it vector} of $\Delta$. By setting $f_{-1}=1$, the $h$-{\it vector} $h(\Delta)=(h_{0}, h_{1},\ldots, h_{d})$ of $\Delta$ is defined by the formula $$\displaystyle\sum_{0\leq i\leq d}f_{i-1}(t-1)^{d-i}=\displaystyle\sum_{0\leq i\leq d}h_{i}t^{d-i}.$$Moreover, the polynomial $h(\Delta, t)=\sum_{0\leq i\leq d}h_{i}t^{i}$ is called the $h$-{\it polynomial} of $\Delta$. It is well-known that the Hilbert series of $\Bbbk[\Delta]$ can be described as $$F(\Bbbk[\Delta], t)=\dfrac{h(\Delta, t)}{(1-t)^{d}}.$$In particular, we see that $h(\Delta*\Delta^{\prime}, t)=h(\Delta, t)\times h(\Delta^{\prime}, t)$. For the relationship between edge ideals and Stanley--Reisner ideals, we have the following notion called an independence complex. The {\it independence complex} of $G$ is the set of independent sets of $G$, which forms a simplicial complex $\Delta(G)$. It is known that the edge ideal of $G$ coincides with the Stanley--Reisner ideal of $\Delta(G)$ (see, for example, \cite[p. 73, Lemma 31]{v0}). A subset $W$ of  $V(G)$ is called a {\it clique}, if the induced subgraph $G|_{W}$ is a complete graph. If $W$ is maximal with respect to inclusion, then we call $W$ a {\it maximal clique} of $G$. A subset $M$ of $E(G)$ is called a {\it matching}, if no two  edges in $M$ share a common vertex. Then the {\it matching number} of $G$, denoted by ${\rm m}(G)$ is defined by $${\rm m}(G)=\max\{|M|\,\,: M\mbox{ is a matching of }G\}.$$ Moreover, if there is no edge in $E(G)\setminus M$ that is contained in the union of edges of $M$, then $M$ is called an {\it induced matching}. Then, the {\it induced matching number} of $G$, denoted by ${\rm im}(G)$, is defined as $${\rm im}(G)=\max\{|M|\,\,: M\mbox{ is an induced matching of }G\}.$$

\subsection{Algebraic background}
Let $S=\Bbbk[x_{1},\ldots, x_{n}]$ be a polynomial ring in $n$ variables over an arbitrary field $\Bbbk$ and let $M$ be a finitely generated graded $S$-module. Then, $M$ admits a {\it graded minimal free resolution} of the form $$0\rightarrow\bigoplus_{j\in\mathbb{Z}}S(-j)^{\beta_{p,j}(M)}\rightarrow\cdots\rightarrow\bigoplus_{j\in\mathbb{Z}}S(-j)^{\beta_{0,j}(M)}\rightarrow M\rightarrow 0,$$
where $S(-j)$ is the graded $S$-module with grading $S(-j)_{k}=S_{-j+k}$. The number $\beta_{i,j}(M)$ is called the $(i,j)$-th {\it graded Betti number} of $M$. The {\it Castelnuovo--Mumford regularity} of $M$ is defined by $${\rm reg}\hspace{0.05cm}M=\max\{j-i\,\,:\beta_{i,j}(M)\neq0\}.$$Also, the {\it projective dimension} of $M$ is defined by $${\rm pd}\hspace{0.05cm}M=\max\{i\,\,:\beta_{i,j}(M)\neq0\mbox{ for some }j\}.$$It is known that the graded Betti numbers can be described in terms of simplicial homology for Stanley--Reisner rings, as shown by Hochster \cite[Theorem 5.1]{h}. Also, they can be described in terms of Alexander dual complex by using the result of Eagon and Reiner known as ``local Alexander duality''\cite[Proposition 1]{er}. 

\begin{thm}[\cite{h, er}]
For a simplicial complex $\Delta$ on the vertex set $[n]$, the $(i,j)$-th graded Betti number of $\Bbbk[\Delta]$ is given by 
$$\beta_{i,j}(\Bbbk[\Delta])=\displaystyle \sum_{\substack{W\subset[n], |W|=j}}{\rm dim}_{\Bbbk}\widetilde{H}_{j-i-1}(\Delta_{W};\Bbbk)$$Moreover, by using ``local Alexander duality'', we have $$\beta_{i,j}(\Bbbk[\Delta])=\displaystyle \sum_{\substack{F\in\Delta^{\vee}, |F|=n-j}}{\rm dim}_{\Bbbk}\widetilde{H}_{i-2}({\rm link}_{\Delta^{\vee}}F;\Bbbk)$$
\end{thm}

Moreover, it is known that the description of the Hilbert series of local cohomology modules of Stanley--Reisner rings by Hochster. 

\begin{thm}[\cite{h}]
For a simplicial complex $\Delta$ on the vertex set $[n]$, then 
\begin{align*}
F(H^j_{\mathfrak{m}}(\Bbbk[\Delta]),t)
& =\displaystyle\sum_\ell\dim_\Bbbk[H^j_{\mathfrak{m}}(\Bbbk[\Delta] )]_\ell \ t^\ell \\[0.15cm]
& =\displaystyle \sum_{F \in \Delta} \dim_\Bbbk\widetilde{H}_{j-|F|-1} ({\rm link}_\Delta F; \Bbbk)\left(\frac{t^{-1}}{1-t^{-1}}\right)^{|F|},
\end{align*}
where $\mathfrak{m}$ is the unique homogeneous maximal ideal of $\Bbbk[x_{1},\ldots, x_{n}]$. 
\end{thm}

These two results provide a relationship between the local cohomology modules and the graded Betti numbers of a Stanley--Reisner ring via the Alexander dual complex.

\begin{cor}\label{hilb}
For a simplicial complex $\Delta$ on the vertex set $[n]$, we have 
\begin{align*}
F(H^j_{\mathfrak{m}}(\Bbbk[\Delta]),t)
& =\displaystyle\sum_\ell\dim_\Bbbk[H^i_{\mathfrak{m}}(\Bbbk[\Delta] )]_\ell \ t^\ell \\[0.15cm]
& =\displaystyle\sum_{i}\beta_{i, n-j+i-1}(\Bbbk[\Delta^{\vee}])\left(\frac{t^{-1}}{1-t^{-1}}\right)^{j-i+1} \\[0.15cm]
& =\displaystyle\sum_{i}\beta_{i, n-j+i}(I_{\Delta^{\vee}})\left(\frac{t^{-1}}{1-t^{-1}}\right)^{j-i}.
\end{align*}
\end{cor}


\section{minimal free resolutions of the cover ideals \\ of clique-whiskered graphs}\label{resol of clique}
In this section, we give the Cohen--Macaulay type of the edge ideals of clique-whiskered graphs and construct an explicit minimal free resolution of the cover ideals of clique-whiskered graphs. As a corollary, we give a formula for the Hilbert series of the local cohomology modules of the residue rings modulo the its edge ideals. First, let us recall the definition of clique-whiskered graphs, which were introduced by Cook and Nagel in the context of combinatorial commutative algebra. 

\begin{defi}\cite[Section 3]{cn}\label{clique-whisker}
Let $G$ be a graph and let $\pi=\{W_{1},\ldots, W_{r}\}$ be a clique partition, namely, $W_{i}$ is a clique of $G$ for all $i$ and  their disjoint union forms $V(G)$. We write $W_{i}=\{x_{i,1},\ldots, 
x_{i,{|W_{i}|}}\}$. Then the graph $G^{\pi}$ on the vertex set $V(G)\cup\{w_{1},\ldots, w_{r}\}$ with the edge set$$E(G)\cup\left(\bigcup_{1\leq i\leq r}\{\{x_{i,j}, w_{i}\}\,\,: x_{i,j}\in W_{i}\}\right)$$
is called the {\it clique-whiskered graph} of $G$ with respect to $\pi$. 
\end{defi}

To determine the Cohen--Macaulay type of $I(G^{\pi})$ and to construct an explicit minimal free resolution of $J(G^{\pi})$, we first introduce some notation. 

\begin{setup}\label{split}
Let $G$ be a graph and let $\pi=\{W_{1},\ldots, W_{r}\}$ be a clique partition. We write $W_{i}=\{x_{i,1},\ldots, x_{i,|W_{i}|}\}$ for all $1\leq i \leq t$ and $W(G)$ for the whisker graph associated with $G$. Here let $y_{i,j}$ be the new vertex attached to $x_{i,j}$ in $W(G)$. For a minimal vertex cover $C$ of $G^{\pi}$, we set $\mathcal{D}(G^{\pi};C)=V(G)\setminus C$. Also, let $$S=\Bbbk[\{x_{i,j}\,\,: 1\leq i\leq n\mbox{ and }1\leq j\leq |W_{i}|\}\cup\{w_{i}\,\,: 1\leq i\leq r\}]$$ be the polynomial ring and let $$S^{\prime}=\Bbbk[\{x_{i,j}, y_{i,j}\,\,: 1\leq i\leq n\mbox{ and }1\leq j\leq |W_{i}|\}$$ be the polynomial ring. Notice that $I(G^{\pi})$, $J(G^{\pi})$ are ideals of $S$ and $I(W(G))$, $J(W(G))$ are ideals of $S^{\prime}$. 
\end{setup}

\begin{ex}
Let $G$ be a graph on the vertex set $\{x_{1,1}, x_{1,2}, x_{1,3}, x_{2,1}\}$ with the edge set $E(G)=\{\{x_{1,1}, x_{1,2}\}, \{x_{1,1}, x_{1,3}\}, \{x_{1,1}, x_{2,1}\}, \{x_{1,2}, x_{1,3}\}, \{x_{1,3}, x_{2,1}\}\}$. Let $\pi=\{\{x_{1,1}, x_{1,2}, x_{1,3}\}, \{x_{2,1}\}\}$ be a clique-partition of $G$.
Then $G$, $G^{\pi}$ and $W(G)$ are
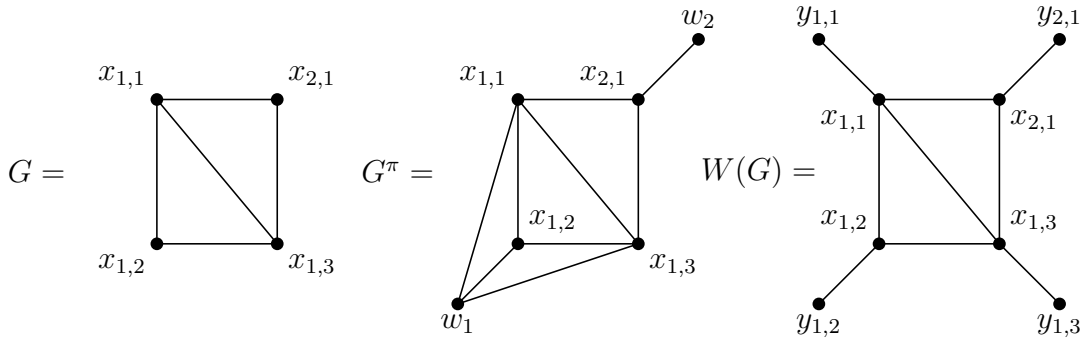
\begin{center}
\begin{tikzpicture}[scale=0.8]
\node at (-11,0) {$G=$};
\coordinate (A1) at (-9,1.2);
\coordinate (A2) at (-7,1.2);
\coordinate (A3) at (-9,-1.2);
\coordinate (A4) at (-7,-1.2);
\fill (A1) circle (3pt);
\fill (A2) circle (3pt);
\fill (A3) circle (3pt);
\fill (A4) circle (3pt);
\node[above left] at (A1) {$x_{1,1}$};
\node[above right] at (A2) {$x_{2,1}$};
\node[below left] at (A3) {$x_{1,2}$};
\node[below right] at (A4) {$x_{1,3}$};
\draw [semithick] (A1)--(A2);
\draw [semithick] (A1)--(A3);
\draw [semithick] (A3)--(A4);
\draw [semithick] (A2)--(A4);
\draw [semithick] (A1)--(A4);
\node at (-5,0) {$G^{\pi}=$};
\coordinate (B1) at (-3,1.2);
\coordinate (B2) at (-1,1.2);
\coordinate (B3) at (-1,-1.2);
\coordinate (B4) at (-3,-1.2);
\coordinate (B5) at (-4,-2.2);
\coordinate (B6) at (0,2.2);
\foreach \pt in {B1,B2,B3,B4,B5,B6}
\fill (\pt) circle (3pt);
\node[above left] at (B1) {$x_{1,1}$};
\node[above left] at (B2) {$x_{2,1}$};
\node[above right] at (B4) {$x_{1,2}$};
\node[below right] at (B3) {$x_{1,3}$};
\node[below] at (B5) {$w_{1}$};
\node[above] at (B6) {$w_{2}$};
\draw [semithick] (B1)--(B2);
\draw [semithick] (B1)--(B4);
\draw [semithick] (B4)--(B3);
\draw [semithick] (B2)--(B3);
\draw [semithick] (B1)--(B3);
\draw [semithick] (B1)--(B5);
\draw [semithick] (B4)--(B5);
\draw [semithick] (B3)--(B5);
\draw [semithick] (B2)--(B6);
\node at (1,0) {$W(G)=$};
\coordinate (C1) at (3,1.2);
\coordinate (C2) at (5,1.2);
\coordinate (C3) at (3,-1.2);
\coordinate (C4) at (5,-1.2);
\coordinate (C5) at (2,-2.2);
\coordinate (C6) at (6,2.2);
\coordinate (C7) at (2,2.2);
\coordinate (C8) at (6,-2.2);
\foreach \pt in {C1,C2,C3,C4,C5,C6,C7,C8}
\fill (\pt) circle (3pt);
\node[below left] at (C1) {$x_{1,1}$};
\node[below right] at (C2) {$x_{2,1}$};
\node[above left] at (C3) {$x_{1,2}$};
\node[above right] at (C4) {$x_{1,3}$};
\node[below] at (C5) {$y_{1,2}$};
\node[above] at (C6) {$y_{2,1}$};
\node[above] at (C7) {$y_{1,1}$};
\node[below] at (C8) {$y_{1,3}$};
\draw [semithick] (C1)--(C2);
\draw [semithick] (C1)--(C3);
\draw [semithick] (C1)--(C4);
\draw [semithick] (C3)--(C4);
\draw [semithick] (C2)--(C4);
\draw [semithick] (C2)--(C6);
\draw [semithick] (C1)--(C7);
\draw [semithick] (C3)--(C5);
\draw [semithick] (C4)--(C8);
\end{tikzpicture}
\end{center}
\end{ex}

In order to give the Cohen--Macaulay type of $I(G^{\pi})$, we present the following relationship between the edge ideals of clique-whiskered graphs and those of whisker graphs. 

\begin{thm}\label{regular}
Assume Setup \ref{split}. We consider the following sequence $${\bf y}=y_{1,2}-y_{1,1}, y_{1,3}-y_{1,1},\ldots, y_{1,|W_{1}|}-y_{1,1},\ldots, y_{r,2}-y_{r,1},\ldots, y_{r,|W_{r}|}-y_{r,1},$$where if $|W_{i}|=1$, then we remove $y_{i, 2}-y_{i,1},\ldots, y_{i,|W_{i}|}-y_{i,1}$ from ${\bf y}$. Then ${\bf y}$ is a regular sequence of $S^{\prime}/I(W(G))$.
\end{thm}
\begin{proof}
Notice that ${\bf y}$ consists of $|V(G)|-r$ elements. 
Now, we have an isomorphism $$S/(I(G^{\pi}))\simeq S^{\prime}/(I(W(G))+({\bf y})).$$
Since $\dim S/I(G^{\pi})=|V(G)|+r-|V(G)|=r$, we obtain that $$\dim S^{\prime}/(I(W(G))+({\bf y}))=r=\dim S^{\prime}/I(W(G))-(|V(G)|-r),$$ and hence ${\bf y}$ is a linear system of parameters. Since $S/I(G^{\pi})$ is Cohen--Macaulay by \cite[Corollary 3.5]{cn}, there exists a regular sequence which contains ${\bf y}$ in $S^{\prime}/I(W(G))$. Therefore since ${\bf y}$ is a subsequence of some regular sequence of $S^{\prime}/I(W(G))$, ${\bf y}$ is a regular sequence, as required. 
\end{proof}

\begin{cor}\label{betti}
Assume Setup \ref{split}. Then, we have 
$$\beta_{i,j}(S/I(G^{\pi}))=\beta_{i,j}(S^{\prime}/I(W(G)))\mbox{ for all }i,j.$$
\end{cor}

\begin{cor}\label{CM type}
Assume Setup \ref{split}. Then the Cohen--Macaulay type of $I(G^{\pi})$ is equal to 
$${\rm type}(S/I(G^{\pi}))=|\{C\,\,: C\mbox{ is a minimal vertex cover of }G\}|.$$
\end{cor}
\begin{proof}
By Corollary \ref{betti}, we have ${\rm type}(S/I(G^{\pi}))={\rm type}(S^{\prime}/I(W(G)))$. From \cite[Corollary 4.4]{crt}, ${\rm type}(S^{\prime}/I(W(G)))=|\{C\,\,: C\mbox{ is a minimal vertex cover of }G\}|$, as required. 
\end{proof}

In the rest of this section, we construct an explicit minimal free resolution of $J(G^{\pi})$. To this end, we present several lemmas. First, we establish the following relationship between the minimal vertex covers of $G^{\pi}$ and those of $W(G)$. 

\begin{lemma}\label{min}
Assume Setup \ref{split}. Then there exists a one-to-one correspondence between the set of minimal vertex covers of $G^{\pi}$ and that of $W(G)$. 
\end{lemma}
\begin{proof}
For a minimal vertex cover $C$ of $G^{\pi}$, we set$$C^{\prime}=(C\setminus\{w_{1},\ldots, w_{r}\})\cup\left(\bigcup_{1\leq i\leq r}\{y_{i,j}\,\,: 1\leq j\leq |W_{i}|\mbox{ with }x_{i,j}\notin C\}\right).$$Notice that $C^{\prime}$ is a minimal vertex cover of $W(G)$ and for any $w_{i}\in C$,  $y_{i,j}$ is uniquely determined since $W_{i}$ is a complete graph for all $i$. Also, for a minimal vertex cover of $D^{\prime}$ of $W(G)$, we set$$D=\left(D^{\prime}\setminus\bigcup_{1\leq i\leq r}\{y_{i,j}\,\,: 1\leq j\leq |W_{i}|\}\right)\cup\left(\bigcup_{1\leq i \leq r}\{w_{i}\,\,: y_{i,j}\in D^{\prime}\mbox{ for some }j\}\right).$$
Notice that $D$ is a minimal vertex cover of $G^{\pi}$. With this notation, we define the map $$\varphi:{\rm Min}(G^{\pi})\rightarrow{\rm Min}(W(G))\mbox{ and }\psi: {\rm Min}(W(G))\rightarrow{\rm Min}(G^{\pi})$$by $\varphi(C)=C^{\prime}$ and $\psi(D^{\prime})=D$. Now, we have 
\begin{align*}
\psi\circ\varphi(C)&=\psi\left((C\setminus\{w_{1},\ldots, w_{r}\})\cup(\bigcup_{1\leq i\leq r}\{y_{i,j}\,\,: 1\leq j\leq |W_{i}|\mbox{ with }x_{i,j}\notin C\})\right) \\
&= (C\setminus\{w_{1},\ldots, w_{r}\})\cup\{w_{i}\,\,: y_{i,j}\in \varphi(C)\mbox{ for some } j\} \\
& = C
\end{align*}
Conversely, we have 
\begin{align*}
\varphi\circ\psi(D^{\prime})&=\varphi\left((D^{\prime}\setminus\bigcup_{1\leq i\leq r}\{y_{i,j}\,\,: 1\leq j\leq |W_{i}|\})\cup(\bigcup_{1\leq i \leq r}\{w_{i}\,\,: y_{i,j}\in D^{\prime}\mbox{ for some }j\})\right) \\
&= \left(D^{\prime}\setminus\bigcup_{1\leq i\leq r}\{y_{i,j}\,\,: 1\leq j\leq |W_{i}|\}\right)\cup\left(\bigcup_{1\leq i\leq r}\{y_{i,j}\,\,: x_{i,j}\in\psi(D^{\prime})\}\right) \\
& = D^{\prime}, 
\end{align*}
which completes the proof. 
\end{proof}

In \cite{cf}, Crupi and Ficarra defined the set $\mathcal{C}(W(G);C^{\prime})=\{x_{i,j}\,\,:y_{i,j}\in C^{\prime}\}$ for a minimal vertex cover $C^{\prime}$ of $W(G)$ to study minimal free resolutions of cover ideals of Cohen--Macaulay very well-covered graphs. Also, they denoted by $\left(\begin{smallmatrix}
\mathcal{C}(W(G) ; C^{\prime}) \\ i \end{smallmatrix}\right)$ the set of all subsets of size $i$ of $\mathcal{C}(W(G) ; C^{\prime})$ for all $0\leq i\leq |\mathcal{C}(W(G) ; C^{\prime})|$. With this notation, we have the following lemma. 

\begin{lemma}\label{compo}
Assume Setup \ref{split}. With the notation above, we have $\mathcal{D}(G^{\pi};C)=\mathcal{C}(W(G);C^{\prime}).$
\end{lemma}
For a  subset $C$ of $V(G^{\pi})$, we set $C_{x}=\{(i,j)\,\,: x_{i,j}\in C\}$ and $C_{w}=\{k\,\,: w_{k}\in C\}$. Then we set 
$${\bf z}_{C}={\bf x}_{C_{x}}{\bf w}_{C_{w}},$$
where ${\bf x}_{C_{x}}=\prod_{(i,j)\in C_{x}} x_{i,j}$ and ${\bf w}_{C_{w}}=\prod_{k\in C_{w}}w_{k}$. 

We construct a minimal free resolution of the cover ideal $J(G^{\pi})$. 

\begin{const}\label{const1}
Assume Setup \ref{split}. We consider the following ordering of vertices $G$: 
$$x_{1,1}<x_{1,2}<\cdots<x_{1,|W_{1}|}<x_{2,1}<\cdots<x_{2,|W_{2}|}<\cdots<x_{r,1}<\cdots<x_{r,|W_{r}|}.$$
Under this ordering of vertices of $V(G)$, let $$\mathbb{F}:\cdots\longrightarrow F_{i}\overset{d_{i}}{\longrightarrow} F_{i-1}\overset{d_{i-1}}{\longrightarrow}\cdots \overset{d_{2}}{\longrightarrow}F_{1}\overset{d_{1}}{\longrightarrow}F_{0}\overset{d_{0}}{\longrightarrow}J(G^{\pi})\rightarrow0$$
be the complex
\begin{enumerate}
\item[-]whose $i$-th free module $F_{i}$ has a basis the symbols ${\bf f}(C ; \sigma)$ having multidegree ${\bf z}_{C}{\bf x}_{\sigma}$, where $C\in{\rm Min}(G^{\pi})$ and 
$\sigma\in\left(\begin{smallmatrix}\mathcal{D}(G^{\pi} ; C) \\ i\end{smallmatrix}\right)$ \\
\item[-]and whose differential is given by $d_{0}({\bf f}(C ; \sigma))={\bf z}_{C}$ for $i=0$ and for $i>0$ is defined as follows:$$d_{i}({\bf f}(C ; \sigma))=\displaystyle\sum_{x_{i,j}\in\sigma}(-1)^{\alpha(\sigma;x_{i,j})}[w_{i}{\bf f}((C\setminus\{w_{i}\})\cup\{x_{i,j}\};\sigma\setminus\{x_{i,j}\})-x_{i,j}{\bf f}(C;\sigma\setminus\{x_{i,j}\})],$$
\end{enumerate}
where $\alpha(\sigma;x_{i,j})= |\{x_{p,q}\in\sigma\,\,: x_{p,q}>x_{i,j}\}|.$
\end{const}

\begin{thm}\label{free resol}
The complex $(\mathbb{F},d_{\bullet})$ given in Construction \ref{const1} is a minimal free resolution of $J(G^{\pi})$. 
\end{thm}
\begin{proof}
Notice that $S^{\prime}/J(W(G))+({\bf y})\simeq S/J(G^{\pi})$, where ${\bf y}$ is defined in Lemma \ref{regular}. First, we show that ${\bf y}$ is a regular sequence of $S^{\prime}/J(W(G))$. Let $$\Gamma=\langle u_{1,1},\ldots, u_{1, |W_{1}|-1},\ldots, u_{r,1},\ldots, u_{r,|W_{r}|-1}\rangle$$ be the simplicial complex, where if $|W_{i}|=1$, then we remove $u_{i,1},\ldots, u_{i, |W_{i}-1|}$ from $\Gamma$. We consider the join $\Delta(G^{\pi})*\Gamma$. Fix a non-negative integer $k$. Since $$f_{k-1}(\Delta(G^{\pi}))=\displaystyle\sum_{\ell=0}^{k}f_{\ell-1}(\Delta(G))\begin{pmatrix}r-\ell \\ k-\ell \end{pmatrix}$$and $\Delta(G^{\pi}) * \Gamma$ is the independence complex of the graph obtained from $G^{\pi}$ by adding $(|V(G)| - r)$ isolated vertices, we have
\begin{align*} 
f_{k-1}(\Delta(G^{\pi}*\Gamma))&=\displaystyle\sum_{\ell=0}^{k}f_{\ell-1}(\Delta(G))\begin{pmatrix}
(r-\ell)+(|V(G)|-r) \\ k-\ell\end{pmatrix} \\
&=\displaystyle\sum_{\ell=0}^{k}f_{\ell-1}(\Delta(G))\begin{pmatrix}
|V(G)|-\ell \\ k-\ell
\end{pmatrix} \\
&=f_{k-1}(\Delta(W(G))).
\end{align*}
Hence, we obtain that $h_{k}(\Delta(W(G)))=h_{k}(\Delta(G^{\pi})*\Gamma)$, and thus, $h(\Delta(G^{\pi})^{\vee}*\Gamma)=h((\Delta(G^{\pi})*\Gamma)^{\vee})=h(\Delta(W(G)))$. Therefore, since $\Bbbk[\Delta(G^{\pi})^{\vee}*\Gamma]$ is a polynomial extension of $\Bbbk[\Delta(G^{\pi})]$, we obtain that 
\begin{align*}
F(S^{\prime}/J(W(G)), t) & = \dfrac{h(\Delta(W(G))^{\vee},t)}{(1-t)^{2|V(G)|-2}} \\[0.1cm]
& = \dfrac{h(\Delta(G^{\pi})^{\vee}*\Gamma,t)}{(1-t)^{2|V(G)|-2}} \\[0.1cm]
& = \dfrac{F(S/J(G^{\pi}), t)}{(1-t)^{|V(G)|-r}} \\[0.1cm] 
& = \dfrac{F(S^{\prime}/(J(W(G))+({\bf y})), t)}{(1-t)^{|V(G)|-r}}.
\end{align*}
By \cite[Corollary 3.2]{s0}, this equality implies that ${\bf y}$ is a regular sequence of $S^{\prime}/J(W(G))$. Therefore, the minimal free resolution of $S/J(G^{\pi})$ is the induced resolution of the minimal free resolution of $S^{\prime}/J(W(G))$. By \cite[Theorem 3.5]{cf}, the minimal free resolution of $J(W(G))$ is$$\mathbb{F}^{W(G)}:\cdots\longrightarrow F_{i}^{W(G)}\overset{d_{i}^{W(G)}}{\longrightarrow} F_{i-1}^{W(G)}\overset{d_{i-1}^{W(G)}}{\longrightarrow}\cdots \overset{d_{1}^{W(G)}}{\longrightarrow}F_{0}^{W(G)}\overset{d_{0}^{W(G)}}{\longrightarrow}J(W(G))\rightarrow0,$$where 
\begin{enumerate}
\item[-] whose $i$-th free module $F_{i}$ has a basis the symbols ${\bf f}(C^{\prime} ; \sigma)$ having multidegree ${\bf z}_{C^{\prime}}{\bf x}_{\sigma}$, where $C^{\prime}\in\mathcal{C}(W(G))$ and 
$\sigma\in\left(\begin{smallmatrix}\mathcal{C}(W(G) ; C^{\prime}) \\ i\end{smallmatrix}\right)$ \\
\item[-]and whose differential is given by $d_{0}^{W(G)}({\bf f}(C^{\prime} ; \sigma))={\bf z}_{C^{\prime}}$ for $i=0$ and for $i>0$ is defined as follows:$$d_{i}({\bf f}(C^{\prime} ; \sigma))=\displaystyle\sum_{x_{i,j}\in\sigma}(-1)^{\alpha(\sigma;x_{i,j})}[y_{i,j}{\bf f}((C^{\prime}\setminus\{y_{i,j}\})\cup\{x_{i,j}\};\sigma\setminus\{x_{i,j}\})-x_{i,j}{\bf f}(C^{\prime};\sigma\setminus\{x_{i,j}\})],$$
\end{enumerate}
where $\alpha(\sigma;x_{i,j})= |\{x_{p,q}\in\sigma\,\,: x_{p,q}>x_{i,j}\}|.$ By identifying $y_{i,j}$ with $w_{i}$ for all $i$ and Lemmas \ref{min}, \ref{compo},  we obtain that the minimal free resolution of $J(G^{\pi})$ is the complex $(\mathbb{F},d_{\bullet})$, as required. 
\end{proof}

As corollaries, we give graded Betti numbers, the projective dimension and the regularity of clique-whiskered graphs. 

\begin{cor}\label{Betti}
Assume Setup \ref{split}. Then we have 
\begin{enumerate}
\item $\beta_{i,i+|V(G)|}(J(G^{\pi}))=\beta_{i}(J(G^{\pi}))=\displaystyle\sum_{C\in{\rm Min}(G^{\pi})}
\begin{pmatrix}
|V(G)\setminus C| \\ i
\end{pmatrix}
\mbox{ for all }i,$
\item ${\rm pd}(J(G^{\pi}))=\max\{|V(G)\setminus C|\,\,:C\in{\rm Min}(G^{\pi})\}.$
\end{enumerate}
In particular, the graded Betti numbers of $J(G^{\pi})$ do not depend upon the characteristic of the underlying field $\Bbbk$. 
\end{cor}

Thanks to Corollary \ref{Betti}, we can easily prove the result of \cite[Corollary 3.6]{lj}. 

\begin{cor}[\cite{lj}, Corollary 3.6]\label{reg of clique-whisker}
Let $G$ be a graph and let $\pi$ be a clique-partition. Then we have ${\rm reg}(S/I(G^{\pi}))={\rm im}(G^{\pi}).$
\end{cor}
\begin{proof} 
It is known that ${\rm reg}(S/I(G^{\pi}))\geq{\rm im}(G^{\pi})$ by \cite[Lemma 2.2]{k}. 
Let $C$ be a minimal vertex cover of $G^{\pi}$ such that $|V(G)\setminus C|={\rm pd}(J(G^{\pi}))$. Since $C$ is a vertex cover of $G$, $V(G)\setminus C$ is an independent set of $G$. From \cite[Corollary 0.3]{t}, we have 
$${\rm reg}(S/I(G^{\pi}))={\rm pd}(J(G^{\pi}))=|V(G)\setminus C|\leq{\rm im}(G^{\pi}),$$
which completes the proof. 
\end{proof}

Finally, we give a formula for the Hilbert series of the local cohomology modules of the residue rings modulo the edge ideals of clique-whiskered graphs by Corollaries \ref{hilb} and \ref{Betti}. Notice that clique-whiskered graphs are Cohen--Macaulay and $\dim S/I(G^{\pi})=r$, where $r$ is the number of elements of the clique-partition. 

\begin{cor}
Let $G$ be a graph and let $\pi$ be a clique-partition. Then the Hilbert series of $H^{r}_{\mathfrak{m}}(S/I(G^{\pi}))$ is given by the following formula, where $r$ is the number of elements of $\pi$: 
\begin{align*}
F(H^r_{\mathfrak{m}}(S/I(G^{\pi})),t)
& =\displaystyle\sum_j\dim_\Bbbk[H^r_{\mathfrak{m}}(S/I(G^{\pi}))]_{\ell} \ t^{\ell} \\[0.15cm]
& =\displaystyle\sum_{i}\displaystyle\sum_{C\in{\rm Min}(G^{\pi})}
\begin{pmatrix}
|V(G)\setminus C| \\ i
\end{pmatrix}\left(\frac{t^{-1}}{1-t^{-1}}\right)^{r-i}.
\end{align*}
\end{cor}


\section{Applications to Cohen--Macaulay chordal graphs, clique corona graphs and Cohen--Macaulay Cameron--Walker graphs}
In this section, we investigate Cohen--Macaulay chordal graphs, clique corona graphs and Cohen--Macaulay Cameron--Walker graphs by using Theorem \ref{free resol}. First, we treat Cohen--Macaulay chordal graphs. For a Cohen--Macaulay chordal graphs, it is known that the following characterization by Herzog, Hibi and Zheng. 

\begin{thm}[\cite{hhz}, Theorem 2.1]
Let $G$ be a chordal graph and let $F_{1},\ldots, F_{m}$ be maximal cliques which admit a free vertex. 
Then the following conditions are equivalent: 
\begin{enumerate}
\item $G$ is Cohen--Macaulay, 
\item $G$ is unmixed, 
\item $V(G)$ is the disjoint union of $F_{1},\ldots, F_{m}$. 
\end{enumerate}
\end{thm}

By using this characterization, we have the following lemma. 

\begin{lemma}\label{chordal}
Let $G$ be a Cohen--Macaulay chordal graph and let $F_{1},\ldots, F_{m}$ be maximal cliques which admit a free vertex. Then we have $$G=(G-\{w_{1},\ldots, w_{m}\})^{\pi},$$ where each $w_{i}$ is a free vertex contained in $F_{i}$ and $\pi=\{F_{i}\setminus\{w_{i}\}\,\,: 1\leq i\leq m\}.$
\end{lemma} 

Thanks to Theorem \ref{free resol} and Lemma \ref{chordal}, we give an explicit minimal free resolution of 
the cover ideals of Cohen--Macaulay chordal graphs. 

\begin{thm}\label{min free of chordal}
Let $G$ be a Cohen--Macaulay chordal graph. Then the complex $(F, d_{\bullet})$ given in Construction \ref{const1} is a minimal free resolution of $J(G)$ under the assumption on Lemma \ref{chordal}. 
\end{thm}

We give corollaries for Cohen--Macaulay chordal graphs by using Theorem \ref{min free of chordal}. 

\begin{cor}
Let $G$ be a Cohen--Macaulay chordal graph and let $F_{1},\ldots, F_{m}$ be maximal cliques which admit a free vertex, $w_{i}$ be a free vertex contained in $F_{i}$ for all $i$. Then we have 
\begin{enumerate}
\item $\beta_{i}(J(G))=\displaystyle\sum_{C\in{\rm Min}(G)}\begin{pmatrix}|V(G)\setminus(\{w_{1},\ldots, w_{m}\}\cup C)| \\ i\end{pmatrix}\mbox{ for all }i,$\item ${\rm pd}(J(G^{\pi}))=\max\{|V(G)\setminus\{w_{1},\ldots, w_{m}\}\cup C)|\,\,:C\in{\rm Min}(G)\}.$
\end{enumerate}
In particular, the graded Betti numbers of $J(G)$ do not depend upon the characteristic of the underlying field $\Bbbk$
\end{cor}

\begin{cor}
Let $G$ be a Cohen--Macaulay chordal graph and let $F_{1},\ldots, F_{m}$ be maximal cliques which admit a free vertex, $w_{i}$ be a free vertex contained in $F_{i}$ for all $i$, and let $S=\Bbbk[V(G)]$. Then the Hilbert series of $H^{m}_{\mathfrak{m}}(S/I(G))$ is given by the following formula: 
\begin{align*}
F(H^m_{\mathfrak{m}}(S/I(G)),t)&=\displaystyle\sum_j\dim_\Bbbk[H^m_{\mathfrak{m}}(S/I(G))]_{\ell} \ t^{\ell} \\[0.15cm]
& =\displaystyle\sum_{i}\displaystyle\sum_{C\in{\rm Min}(G)}
\begin{pmatrix}
|(V(G)\setminus\{w_{1},\ldots, w_{m}\})\setminus C| \\ i
\end{pmatrix}\left(\frac{t^{-1}}{1-t^{-1}}\right)^{m-i}.
\end{align*}
\end{cor}

Next, we investigate clique corona graphs, which were introduced by Hoang and Pham in the context of combinatorial commutative algebra. Let us recall the definition of clique corona graphs. 

\begin{defi}[\cite{hp}, Section 2]
For a graph $G$ on the vertex set $X_{[h]}=\{x_{1},\ldots, x_{h}\}$, we set$$\mathcal{H}=\{H_{i}\,\,: x_{i}\in V(G)\},$$where $H_{i}$ is a non-empty graph indexed by the vertex $x_{i}$. The {\it corona graph} $G\circ\mathcal{H}$ of $G$ and $\mathcal{H}$ is the disjoint union of $G$ and $H_{i}$, with additional edges joining each vertex $x_{i}$ to all vertices $H_{i}$. In particular, $G\circ\mathcal{H}$ is called the {\it clique corona graph}, if $H_{i}$ is the complete graph for all $i$. 
\end{defi}

By the definition of clique-corona graphs, we have the following lemma. 

\begin{lemma}\label{corona}
Let $G$ be a graph on the vertex set $X_{[h]}=\{x_{1},\ldots, x_{h}\}$ and let $\mathcal{H}=\{H_{i}\,\,: x_{i}\in V(G)\}$. We set $H_{i}=K_{m_{i}}$ on the vertex set $\{y_{i,1},\ldots, y_{i,m_{i}}\}$. Then we have
$$G\circ\mathcal{H}=(G\circ\mathcal{H}-\{y_{1,m_{1}}, \ldots, y_{h,m_{h}}\})^{\pi},$$
where $\pi=\{(K_{m_{i}}\cup\{x_{i}\})\setminus\{y_{i,m_{i}}\}\,\,: 1\leq i\leq h\}$. 
\end{lemma}

Thanks to Theorem \ref{free resol} and Lemma \ref{corona}, we give an explicit minimal free resolution of 
the cover ideals of clique-corona graphs. 

\begin{thm}\label{min free of corona}
Let $G$ be a graph on the vertex set $X_{[h]}=\{x_{1},\ldots, x_{h}\}$ and let $\mathcal{H}=\{H_{i}\,\,: x_{i}\in V(G)\}$. We set $H_{i}=K_{m_{i}}$ on the vertex set $\{y_{i,1},\ldots, y_{i,m_{i}}\}$. Then the complex $(F, d_{\bullet})$ given in Construction \ref{const1} is a minimal free resolution of $J(G\circ\mathcal{H})$ under the assumption on Lemma \ref{corona}. 
\end{thm}

We give corollaries for clique corona graphs by using Theorem \ref{min free of corona}. 

\begin{cor}\label{clique-Betti}
Let $G$ be a graph on the vertex set $X_{[h]}=\{x_{1},\ldots, x_{h}\}$ and let $\mathcal{H}=\{H_{i}\,\,: x_{i}\in V(G)\}$. We set $H_{i}=K_{m_{i}}$ on the vertex set $\{y_{i,1},\ldots, y_{i,m_{i}}\}$. Then we have 
\begin{enumerate}
\item $\beta_{i}(J(G\circ\mathcal{H}))=\displaystyle\sum_{C\in{\rm Min}(G\circ\mathcal{H})}
\begin{pmatrix}
|V(G\circ\mathcal{H})\setminus(C\cup\{y_{1,m_{1}},\ldots, y_{h,m_{h}}\})| \\ i
\end{pmatrix}
\mbox{ for all }i,$
\item ${\rm pd}(J(G))=\max\{|V(G\circ\mathcal{H})\setminus(C\cup\{y_{1,m_{1}},\ldots, y_{h,m_{h}}\})|\,\,:C\in{\rm Min}(G\circ\mathcal{H})\}.$
\end{enumerate}
In particular, the graded Betti numbers of $J(G\circ\mathcal{H})$ do not depend upon the characteristic of the underlying field $\Bbbk$. 
\end{cor}

We obtain the following theorem by Corollary \ref{reg of clique-whisker} and Lemma \ref{corona}. 

\begin{cor}[\cite{hp}, Theorem 3.5]
Let $G$ be a graph on the vertex set $X_{[h]}$ and let $\mathcal{H}=\{H_{i}\,\,: x_{i}\in V(G)\}$. If $G\circ\mathcal{H}$ is the clique corona graph, then
$${\rm reg}(S/I(G\circ\mathcal{H}))={\rm im}(G\circ\mathcal{H}).$$
\end{cor}

\begin{cor}
Let $G$ be a graph on the vertex set $X_{[h]}$ and let $\mathcal{H}=\{H_{i}\,\,: x_{i}\in V(G)\}$, $S=\Bbbk[V(G\circ\mathcal{H})]$. If $G\circ\mathcal{H}$ is the clique corona graph, then the Hilbert series of $H^{h}_{\mathfrak{m}}(S/I(G\circ\mathcal{H}))$ is given by the following formula: 
\begin{align*}
&F(H^h_{\mathfrak{m}}(S/I(G\circ\mathcal{H})),t) \\[0.15cm]
& =\displaystyle\sum_j\dim_\Bbbk[H^h_{\mathfrak{m}}(S/I(G\circ\mathcal{H}))]_{\ell} \ t^{\ell} \\[0.15cm]
&=\displaystyle\sum_{i}\displaystyle\sum_{C\in{\rm Min}(G)}
\begin{pmatrix}
|V(G\circ\mathcal{H})\setminus(C\cup\{y_{1,m_{1}},\ldots, y_{h,m_{h}}\})| \\ i
\end{pmatrix}\left(\frac{t^{-1}}{1-t^{-1}}\right)^{h-i}.
\end{align*}
\end{cor}

To give corollaries for Cameron--Walker graphs, let us first recall the definition of Cameron--Walker graphs. According to \cite[Theorem 1]{cw} and \cite[Remark 0.1]{hhko}, for a graph $G$, the equality ${\rm im}(G) = {\rm m}(G)$ holds if and only if $G$ is one of the following graphs: 
\begin{enumerate}
\item a star graph, or 
\item a star triangle, or   
\item a connected finite graph consisting of a connected bipartite graph with vertex partition $\{x_{1}, \ldots, x_{n}\} \cup \{y_{1}, \ldots, y_{m}\}$ such that there is at least one leaf edge attached to each vertex $x_{i}$ and that there may be possibly some pendant triangles attached to each vertex $y_{j}$.
\end{enumerate} 

\begin{defi}
A finite connected simple graph $G$ is said to be a {\em Cameron--Walker graph}, 
if ${\rm im}(G) = {\rm m}(G)$ and if $G$ is neither a star graph nor a star triangle. 
\end{defi}

It is known that the following characterization of Cohen--Macaulay Cameron--Walker graphs by Kimura, Hibi, Higashitani and O' Keefe. 

\begin{thm}[\cite{hhko}, Theorem 1.3]
For a Cameron--Walker graph $G$, 
the following five conditions are equivalent: 
\begin{enumerate}
\item $G$ is unmixed. 
\item $G$ is Cohen--Macaulay. 
\item $G$ is unmixed and shellable. 
\item $G$ is unmixed and vertex decomposable. 
\item $G$ consists of a connected bipartite graph with vertex partition 
$[n] \sqcup [m]$ such that 
there is exactly one leaf edge attached to each vertex
$i \in [n]$ and that there is exactly one pendant triangle
attached to each vertex $j \in [m]$.
\end{enumerate}
\end{thm}

Using this characterization, we introduce some notation and state a lemma.

 \begin{nota}\label{CW-nota}
Let $G$ be a Cohen--Macaulay Cameron--Walker graph. Let denote $G_{\rm bip}$ the bipartite part of $G$, namely,  $G_{\rm bip}$ is the induced subgraph of $G$ on $\{x_{1},\ldots, x_{n}\}\cup\{y_{1},\ldots, y_{m}\}$. Also, let $w_{i}$ be the free vertex attached to $x_{i}$ for all $i$ and $\{z_{j}, w_{n+j}\}$ and $z_{j}, w_{n+j}$ be the pendant attached to $y_{j}$ for all $j$. 
\end{nota}

\begin{lemma}\label{CW}
Let $G$ be a Cohen--Macaulay Cameron--Walker graph. Assume the condition in Notation \ref{CW-nota}. Then, we have $$G=(G|_{V(G_{\rm bip})\cup\{z_{1},\ldots, z_{m}\}})^{\pi},$$where $\pi=\{x_{i}\,\,: 1\leq i\leq n\}\cup\{\{y_{j}, z_{j}\}\,\,: 1\leq j\leq m\}.$
\end{lemma}

Thanks to Theorem \ref{free resol} and Lemma \ref{CW}, we give an explicit minimal free resolution of 
the cover ideals of Cohen--Macaulay Cameron--Walker graphs. 

\begin{thm}\label{min free of CW}
Let $G$ be a Cohen--Macaulay Cameron--Walker graph. Then the complex $(F, d_{\bullet})$ given in Construction \ref{const1} is a minimal free resolution of $J(G)$ under the assumption on Lemma \ref{CW}. 
\end{thm}

We give corollaries for Cohen--Macaulay Cameron--Walker graphs by using Theorem \ref{min free of CW}. 

\begin{cor}
Let $G$ be a Cohen--Macaulay Cameron--Walker graph. Assume the condition in Notation \ref{CW-nota}. Then, we have 
\begin{enumerate}
\item $\beta_{i}(J(G))=\displaystyle\sum_{C\in{\rm Min}(G)}\begin{pmatrix}|V(G)|_{V(G_{\rm bip})\cup\{z_{1},\ldots, z_{m}\}}\setminus C| \\ i\end{pmatrix}\mbox{ for all }i,$
\item ${\rm pd}(J(G))=\max\{|V(G|_{V(G_{\rm bip})\cup\{z_{1},\ldots, z_{m}\}})\setminus C|\,\,:C\in{\rm Min}(G)\}.$
\end{enumerate}
In particular, the graded Betti numbers of $J(G)$ do not depend upon the characteristic of the underlying field $\Bbbk$. 
\end{cor}

\begin{cor}
Let $G$ be a Cohen--Macaulay Cameron--Walker graph and let $S=\Bbbk[V(G)]$. Then the Hilbert series of $H^{n+m}_{\mathfrak{m}}(S/I(G))$ is given by the following formula under the assumptions of Notation \ref{CW-nota}: 
\begin{align*}
F(H^{n+m}_{\mathfrak{m}}(S/I(G)),t)&=\displaystyle\sum_j\dim_\Bbbk[H^{n+m}_{\mathfrak{m}}(S/I(G))]_j \ t^j \\[0.15cm]
& =\displaystyle\sum_{i}\displaystyle\sum_{C\in{\rm Min}(G)}
\begin{pmatrix}
|V(G)|_{V(G_{\rm bip})\cup\{z_{1},\ldots, z_{m}\}}\setminus C| \\ i
\end{pmatrix}\left(\frac{t^{-1}}{1-t^{-1}}\right)^{n+m-i}.
\end{align*}
\end{cor}


\section{multi-clique-whiskered graphs as a generalization of clique-whiskered graphs}\label{multi}
In this section, we define multi-clique-whiskered graphs as a generalization of clique-whiskered graphs. Also, this graph is a generalization of multi-whisker graphs which were first introduced in \cite[Section 1]{mpt} in the context of combinatorial commutative algebra. We show that multi-clique-whiskered graphs are sequentially Cohen--Macaulay as a generalization of a results of \cite[Theorem 3.3]{cn} and \cite[Theorem 6.1]{mpt}. This result leads to a characterization of the projective dimension of multi-clique-whiskered graphs. Also, we construct minimal free resolutions of the cover ideals of multi-clique-whiskered graphs and very well-covered graphs. As a corollary, we obtain the regularity of the edge ideals of multi-clique-whiskered graphs and the Hilbert series of the local cohomology modules of  residue rings modulo edge ideals associated with these graphs. Let us define the multi-clique-whiskered graphs as a generalization of clique-whiskered graphs. 

\begin{defi}
Let $G$ be a graph and let $\pi=\{W_{1},\ldots, W_{r}\}$ be a clique partition. Also we let  $n_{1},\ldots, n_{r}$ be positive integers. We write $W_{i}=\{x_{i,1},\ldots, x_{i,{|W_{i}|}}\}$. Then the graph $G^{\pi}[n_{1},\ldots, n_{r}]$ on the vertex set $V(G)\cup\{w_{1,1},\ldots, w_{1,n_{1}},\ldots, w_{r,1},\ldots, w_{r,n_{r}}\}$ with the edge set$$E(G)\cup\left(\bigcup_{1\leq i\leq r}\{\{x_{i,j}, w_{i,k}\}\,\,: x_{i,j}\in W_{i}\}\right)$$
is called the {\it multi-clique-whiskered graph} of $G$ with respect to $\pi$. 
\end{defi}

Firstly, we prove that multi-clique-whiskered graphs are vertex decomposable, and hence sequentially Cohen--Macaulay, as a generalization of results on clique-whiskered graphs \cite[Theorem 3.3]{cn} and multi-whisker graphs \cite[Theorem 6.1]{mpt}. Before proceeding, let us recall definitions of them, which are defined in \cite{bw} and \cite{s1}.
For a simplicial complex $\Delta$ on the vertex set $V=\{x_{1},\ldots, x_{n}\}$, $\Delta$ is called {\it vertex decomposable}, if either: 
\begin{enumerate}
\item $\Delta=\langle\{x_{1},\ldots, x_{n}\}\rangle$, or $\Delta=\emptyset$. 
\item There exists a vertex $x\in V$ such that ${\rm link}_{\Delta}(\{x\})$ and ${\rm del}_{\Delta}(\{x\})$ are vertex decomposable, and every facet of ${\rm del}_{\Delta}(\{x\})$ is a facet of $\Delta$,
\end{enumerate}
where $${\rm link}_{\Delta}(\{x\})=\{F\in\Delta\,\,: \{x\}\cap F=\emptyset\mbox{ and }\{x\}\cup F\in\Delta\}$$and$${\rm del}_{\Delta}(\{x\})=\{F\in\Delta\,\,: \{x\}\cap F=\emptyset\}.$$Moreover, for a graded module $M$ over $S=\Bbbk[x_{1},\ldots, x_{n}]$, $M$ is called {\it sequentially Cohen--Macaulay}, if there exists a filtration$$0=M_{0}\subset M_{1}\subset\cdots\subset M_{t}=M$$of $M$ by graded $S$-modules such that $\dim M_{i}/M_{i-1}<\dim M_{i+1}/M_{i}$, and $M_{i}/M_{i-1}$ is Cohen--Macaulay for all $i$. For a graph $G$, we say that $G$ is {\it vertex decomposable} , if the independence complex $\Delta(G)$ is vertex decomposable. Also, $G$ is called {\it sequentially Cohen--Macaulay} if $S/I(G)$ is sequentially Cohen--Macaulay. 

\begin{thm}\label{seq CM}
Let $G$ be a graph, $\pi=\{W_{1},\ldots, W_{r}\}$ be a clique partition and let  $n_{1},\ldots, n_{r}$ be positive integers. We write $W_{i}=\{x_{i,1},\ldots, x_{i,{|W_{i}|}}\}$. Then the multi-clique-whiskered graph $G^{\pi}[n_{1},\ldots, n_{r}]$ is vertex decomposable, and hence sequentially Cohen--Macaulay. 
\end{thm}
\begin{proof}
We procced by induction on $|V(G^{\pi}[n_{1},\ldots, n_{r}])|$. If $|V(G^{\pi}[n_{1},\ldots, n_{r}])|=2$, then $G^{\pi}[n_{1},\ldots, n_{r}]$ is just the complete graph $K_{2}$ and hence it is vertex decomposable. Assume that $|V(G^{\pi}[n_{1},\ldots, n_{r}])|>2$. If $n_{i}=1$ for all $i$, then it is clique-whiskered graph then we are done by \cite[Theorem 1.1]{hhko}. Hence we may assume that there exist an integer $i$ such that $n_{i}>1$. Without loss of generality, we may assume that $i=r$. Then ${\rm link}_{\Delta(G^{\pi}[n_{1},\ldots, n_{r}])}w_{r, n_{r}}$ is the simplicial complex$$\Delta(G^{\pi}[n_{1},\ldots, n_{r}]-\{x_{r,j}, w_{r, n_{r}}\,\,: 1\leq j\leq |W_{r}|\})*\{w_{r, 1}.\ldots, w_{r, n_{r}-1}\}.$$Notice that if $r=$1, then $G^{\pi}[n_{1},\ldots, n_{r}]-\{x_{r,j}, w_{r, n_{r}}\,\,: 1\leq j\leq |W_{r}|\}$ has no vertices, and hence vertex decomposable. Since $G^{\pi}[n_{1},\ldots, n_{r}]-\{x_{r,j}, w_{r, n_{r}}\,\,: 1\leq j\leq |W_{r}|\}$ is a multi-clique-whiskered graph, by the induction hypothesis, this is vertex decomposable. Since a simplicial join preserves the vertex decomposability, ${\rm link}_{\Delta(G^{\pi}[n_{1},\ldots, n_{r}])}w_{r, n_{r}}$ is vertex decomposable. On the other hand, since $${\rm del}_{\Delta(G^{\pi}[n_{1},\ldots, n_{r}])}w_{r,n_{r}}=\Delta(G^{\pi}[n_{1},\ldots, n_{r}]-w_{r,n_{r}})$$ and $G^{\pi}[n_{1},\ldots, n_{r}]-w_{r,|W_{r}|}$ is a multi-clique-whiskered graph, this is vertex decomposable. Moreover, any independent set of $G^{\pi}[n_{1},\ldots, n_{r}]-N[w_{r,|W_{r}|}]$ is not a maximal independent set of $G^{\pi}[n_{1},\ldots, n_{r}]-w_{r,|W_{r}|}$, and hence $\mathcal{F}({\rm del}_{\Delta(G^{\pi}[n_{1},\ldots, n_{r}])}w_{r,|W_{r}|})\subset\mathcal{F}(\Delta(G^{\pi}[n_{1},\ldots, n_{r}]))$. Therefore, $\Delta(G^{\pi}[n_{1},\ldots, n_{r}])$ is vertex decomposable, which completes the proof. 
\end{proof}

\begin{cor}
Let $G$ be a graph, $\pi=\{W_{1},\ldots, W_{r}\}$ be a clique partition and let  $n_{1},\ldots, n_{r}$ be positive integers. Then the projective dimension of $I(G^{\pi}[n_{1},\ldots, n_{r}])$ is equal to 
$${\rm pd}\hspace{0.05cm}\Bbbk[G^{\pi}[n_{1},\ldots, n_{r}]]/I(G^{\pi}[n_{1},\ldots, n_{r}])={\rm bight}\hspace{0.05cm}I(G^{\pi}[n_{1},\ldots, n_{r}]).$$
\end{cor}
\begin{proof}
Notice that $F\cup N(F)=V(G^{\pi}[n_{1},\ldots, n_{r}])$ holds for any maximal independent set $F$ of $G^{\pi}[n_{1},\ldots, n_{r}]$. Therefore, by \cite[Remark 5.7]{ds} and Theorem \ref{seq CM}, the assertion follows. 
\end{proof}

In the rest of this section, we give explicit minimal free resolutions of the cover ideals of multi-clique-whiskered graphs and very well-covered graphs. Firstly, we investigate multi-clique-whiskered graphs as a corollary of Theorem \ref{free resol}. 

\begin{const}\label{const2}
Let $G$ be a graph, $\pi=\{W_{1},\ldots, W_{r}\}$ be a clique partition and let  $n_{1},\ldots, n_{r}$ be positive integers. We write $W_{i}=\{x_{i,1},\ldots, x_{i,{|W_{i}|}}\}$. Let $G^{\pi}[n_{1},\ldots, n_{r}]$ be the multi-clique-whiskered graph with respect to $\pi$. Set $\mathcal{D}(G^{\pi}[n_{1},\ldots, n_{r}]; C)=V(G)\setminus C$ for any minimal vertex cover $C$ of $G^{\pi}[n_{1},\ldots, n_{r}]$. We consider the following ordering of vertices $V(G)$: 
$$x_{1,1}<x_{1,2}<\cdots<x_{1,|W_{1}|}<x_{2,1}<\cdots<x_{2,|W_{2}|}<\cdots<x_{r,1}<\cdots<x_{r,|W_{r}|}.$$
Under this ordering of vertices of $V(G)$, let $$\mathbb{F}:\cdots\longrightarrow F_{i}\overset{d_{i}}{\longrightarrow} F_{i-1}\overset{d_{i-1}}{\longrightarrow}\cdots \overset{d_{2}}{\longrightarrow}F_{1}\overset{d_{1}}{\longrightarrow}F_{0}\overset{d_{0}}{\longrightarrow}J(G^{\pi}[n_{1},\ldots, n_{r}])\rightarrow0$$
be the complex
\begin{enumerate}
\item[-]whose $i$-th free module $F_{i}$ has a basis the symbols ${\bf f}(C ; \sigma)$ having multidegree ${\bf z}_{C}{\bf x}_{\sigma}$, where $C\in{\rm Min}(G^{\pi}[n_{1},\ldots, n_{r}])$ and 
$\sigma\in\left(\begin{smallmatrix}\mathcal{D}(G^{\pi}[n_{1},\ldots, n_{r}] ; C) \\ i\end{smallmatrix}\right)$ \\
\item[-]and whose differential is given by $d_{0}({\bf f}(C ; \sigma))={\bf z}_{C}$ for $i=0$ and for $i>0$ is defined as follows: 
\begin{align*}
&d_{i}({\bf f}(C ; \sigma)) \\
&=\sum_{x_{i,j}}(-1)^{\alpha(\sigma;x_{i,j})}[{\bf w}_{i}{\bf f}((C\setminus\{w_{i,1},\ldots,  w_{i, n_{i}}\})\cup\{x_{i,j}\};\sigma\setminus\{x_{i,j}\})-x_{i,j}{\bf f}(C;\sigma\setminus\{x_{i,j}\})],
\end{align*}
\end{enumerate}
where the sum runs over the variables $x_{i,j}$ in $\sigma$ and $\alpha(\sigma;x_{i,j})= |\{x_{p.q}\in\sigma\,\,: x_{p,q}>x_{i,j}\}|$, ${\bf w}_{i}=w_{i,1}\cdots w_{i, n_{i}}$. 
\end{const}

\begin{thm}\label{min resol2}
The complex $(\mathbb{F},d_{\bullet})$ given in Construction \ref{const2} is a minimal free resolution of $J(G^{\pi}[n_{1},\ldots, n_{r}])$. 
\end{thm}
\begin{proof}
Set $N=\sum_{i=1}^m n_{i}$. We consider a minimal free resolution provided in Theorem \ref{free resol} as a $\mathbb{Z}^{N+r}$–graded resolution. For $i=1,\ldots,m$, we set $I_{i}=\sum_{s<i}n_{s}$. On $S=k[x_1,\dots,x_r,w_1,\dots,w_m]$,  we define
$$\deg(w_i)=\sum_{t=I_i+1}^{I_i+n_i} e_t \in \mathbb{Z}^{N+r}\mbox{ and }
\deg(x_j)=e_{N+j}\in \mathbb{Z}^{N+r},$$
where $e_1,\dots,e_{N+r}$ denote the standard basis vectors of $\mathbb{Z}^{N+r}$. Then, under this grading, a free resolution constructed in Construction \ref{const1} remains a minimal free resolution, as guaranteed by Theorem \ref{free resol}. Replacing $w_{i}$ by ${\bf w}_{i}=w_{i,1}\cdots w_{i, n_{i}}$ for each $i$, one has $J(G^{\pi}[n_{1},\ldots, n_{r}])$. Therefore, by considering a minimal free resolution provided in Construction \ref{const1} as a $\mathbb{Z}^{N+r}$ in $\Bbbk[x_{1},\ldots, x_{r}, w_{1,1},\ldots, w_{1, n_{1}},\ldots, w_{r,n_{r}}]$, we see that a resolution provided in Construction \ref{const2} is a minimal free resolution of $J(G^{\pi}[n_{1},\ldots, n_{r}])$, as required. 
\end{proof}

\begin{cor}\label{betti of multi}
Let $G$ be a graph, $\pi=\{W_{1},\ldots, W_{r}\}$ be a clique partition and let  $n_{1},\ldots, n_{r}$ be positive integers. Then we have 
\begin{enumerate}
\item $\beta_{i, i+j}(J(G^{\pi}[n_{1},\ldots, n_{r}]))=\displaystyle \sum_{\substack{C\in{\rm Min}(G^{\pi}[n_{1},\ldots, n_{r}]), \\[0.1cm] |C|=j}}
\begin{pmatrix}
|V(G)\setminus C| \\ i
\end{pmatrix}\mbox{ for all }i, j$ \\[0.15cm]
\item ${\rm pd}(J(G^{\pi}[n_{1},\ldots, n_{r}]))=\max\{|V(G)\setminus C|\,\,:C\in\mathcal{C}(G^{\pi}[n_{1},\ldots, n_{r}])\}.$
\end{enumerate}
In particular, the graded Betti numbers of $J(G^{\pi}[n_{1},\ldots, n_{r}])$ do not depend upon the characteristic of the underlying field $\Bbbk$. 
\end{cor}

Thanks to Theorem \ref{min resol2}, we obtain the regularity of $I(G^{\pi}[n_{1},\ldots, n_{r}])$ as a generalization of \cite[Corollary 3.6]{lj} and \cite[Corollary 4.4]{mpt}. 

\begin{cor}\label{reg of multi}
Let $G$ be a graph, $\pi=\{W_{1},\ldots, W_{r}\}$ be a clique partition and let  $n_{1},\ldots, n_{r}$ be positive integers, $S=\Bbbk[V(G^{\pi}[n_{1},\ldots, n_{r}])]$. Then the regularity of $I(G^{\pi}[n_{1},\ldots, n_{r}])$ is equal to $${\rm reg}S/I(G^{\pi}[n_{1},\ldots, n_{r}])={\rm im}(G^{\pi}[n_{1},\ldots, n_{r}]).$$
\end{cor}
\begin{proof}
See the proof of Corollary \ref{reg of multi}. 
\end{proof}

\begin{cor}\label{local coho of multi}
Let $G$ be a graph, and let $\pi = \{W_1, \ldots, W_r\}$ be a clique partition of $G$. Let $n_1, \ldots, n_r$ be positive integers, and set $S = \Bbbk[V(G^{\pi}[n_1, \ldots, n_r])]$. Then, for a positive integer $j$, the Hilbert series of $H^j_{\mathfrak{m}}(S/I(G^{\pi}[n_1, \ldots, n_r]))$ is given by the following formula:
\begin{align*}
F(H^j_{\mathfrak{m}}(S/I(G^{\pi}[n_{1},\ldots, n_{r}])),t) 
& =\displaystyle\sum_\ell\dim_\Bbbk[H^j_{\mathfrak{m}}(S/I(G^{\pi}[n_{1},\ldots, n_{r}]))]_{\ell} \ t^{\ell} \\[0.15cm]
& =\displaystyle\sum_{i}\displaystyle\sum_{\substack{C\in{\rm Min}(G^{\pi}[n_{1},\ldots, n_{r}]) \\[0.15cm] |C|=|V(G^{\pi}[n_{1},\ldots, n_{r}])|-j}}
\begin{pmatrix}
|V(G)\setminus C| \\ i
\end{pmatrix}\left(\frac{t^{-1}}{1-t^{-1}}\right)^{j-i}.
\end{align*}
\end{cor}

\begin{rem}
As mentioned above, multi-clique-whiskered graphs are a generalization of multi-whisker graphs. Therefore, for multi-whisker graphs, Corollary \ref{local coho of multi} provides a description of the Hilbert series of the local cohomology modules in a form different from that in \cite[Theorem 4.1]{mpt}.
\end{rem}

Finally, we provide an explicit minimal free resolution of the cover ideal of a very well-covered graph as a corollary of \cite[Theorem 3.5]{cf}. Recall that a graph $G$ is called \emph{very well-covered} if it has no isolated vertices and all minimal vertex covers have the same cardinality, which is equal to half the number of its vertices. Let $H$ be a Cohen--Macaulay very well-covered graph with $2d_0$ vertices. We follow the notation in \cite{kpty}. According to \cite{crt}, we may assume the following condition:

\vspace{0.1cm}
\begin{itemize}
\item[$(\ast)$] $V(H)=X_{[d_0]}\cup Y_{[d_0]}$, where $X_{[d_0]}=\{x_1,\ldots,x_{d_0}\}$ is a minimal vertex cover of $H$ and $Y_{[d_0]}=\{y_1,\ldots,y_{d_0}\}$ is a maximal independent set of $H$ such that $\{x_1y_1,\ldots,x_{d_0}y_{d_0}\} \subseteq E(H)$. Moreover, $x_iy_j \in E(H)$ implies that $i\le j$.
\end{itemize}
To study the structure of non-Cohen--Macaulay very well-covered graphs, a particularly useful result is provided in~\cite[Theorem~3.5]{kpty}. Let $G$ be a graph with $xy\in E(G)$, then, the graph $G^{\prime}$, obtained by replacing the edge $xy$ in $G$ with a complete bipartite graph $K_{i, i}$ is defined as $$V(G')=(V(G)\setminus \{x,y\} )\cup \{x_1,\ldots,x_i\}\cup\{y_1 ,\ldots,y_i\}$$ and  
\[
\begin{array}{lll}
E(G')	& = & E(G_{V(G)\setminus \{x,y\}}) \cup\{x_jy_k\,\,: 1\le j,k\le i\} \\[0.15cm]
& & \hspace{2.65cm}\cup\ \{x_jz\,\,: 1\le j\le i,\ z \in V(G) \setminus \{x,y \},\ xz\in E(G)\} \\[0.15cm]
& & \hspace{2.65cm}\cup\ \{y_jz\,\,: 1\le j\le i,\ z \in V(G) \setminus \{x,y \},\ yz\in E(G)\}.
\end{array}
\]
Let $H$ be a Cohen--Macaulay very well-covered graph with $2d_0$ vertices and assume the condition $(\ast)$. Let $H^{\prime}$ be a graph with the vertex set $$V(H^{\prime})=\bigcup_{i=1}^{d_0}\Big(\big\{ x_{i1},\ldots,x_{in_i}\big\}\cup\big\{y_{i1},\ldots,y_{in_i}\big\}\Big)$$ which is obtained by replacing the edges $x_1y_1,\ldots,x_{d_0}y_{d_0}$ in $H$ with the complete bipartite graphs $K_{n_1,n_1},\ldots,K_{n_{d_0},n_{d_0}}$, respectively. We write $H(n_1,\ldots,n_{d_0})$ for $H^{\prime}$. The following theorem is the result mentioned above.

\begin{thm}[\cite{kpty}, Theorem 3.5]\label{structure}
Let $G$ be a very well-covered graph on the vertex set $X_{[d]}\cup Y_{[d]}$. Then there exist positive integers $n_1,\ldots,n_{d_0}$ with $\sum_{i\in [d_0] }n_i=d$ and a Cohen--Macaulay very well-covered graph $H$ on the vertex set $X_ {[d_0]}\cup Y_{[d_0]}$ such that $G\cong H(n_1,\ldots,n_{d_0})$.
\end{thm}

By Theorem \ref{structure} and \cite[Theorem 3.5]{cf}, we construct a minimal free resolution of the cover ideal of a very well-covered graph $G=H(n_{1},\ldots, n_{d_{0}})$. 

\begin{const}\label{const3}
Let $G$ be a very well-covered graph on the vertex set $X_{[d]}\cup Y_{[d]}$ and let $n_1,\ldots,n_{d_0}$ be positive integers with $\sum_{i\in [d_0] }n_i=d$ and a Cohen--Macaulay very well-covered graph $H$ on the vertex set $X_ {[d_0]}\cup Y_{[d_0]}$ such that $G\cong H(n_1,\ldots,n_{d_0})$. 
We write $X_{[d_{0}]}=\{x^{\prime}_{1},\ldots, x^{\prime}_{d_{0}}\}$ and $Y_{[d_{0}]}=\{y^{\prime}_{1},\ldots, y^{\prime}_{d_{0}}\}$. For all $i=1,\ldots, d_{0}$, we define a degree of $x^{\prime}_{i}$ and $y^{\prime}_{i}$ as
$${\rm deg}(x_{i}^{\prime})={\rm deg}(y^{\prime}_{i})=(0, \ldots, 0, 1, 1,\ldots, 1, 0, \ldots, 0),$$where the $j$-th position of ${\rm deg}(x_{i}^{\prime})$ is 1 for $j=i, i+1,\ldots, i+n_{i}$, and 0 otherwise. Under this grading for variables, let$$\mathbb{F}:\cdots\longrightarrow F_{i}\overset{d_{i}}{\longrightarrow} F_{i-1}\overset{d_{i-1}}{\longrightarrow}\cdots \overset{d_{1}}{\longrightarrow}F_{0}\overset{d_{0}}{\longrightarrow}J(H(n_{1},\ldots, n_{d_{0}}))\rightarrow0,$$where 
\begin{enumerate}
\item[-] whose $i$-th free module $F_{i}$ has a basis the symbols ${\bf f}(C ; \sigma)$ having multidegree ${\bf z}_{C}{\bf x}^{\prime}_{\sigma}$, where $C\in{\rm Min}(H)$ and 
$\sigma\in\left(
\begin{smallmatrix}
\mathcal{C}(H ; C) \\ i
\end{smallmatrix}
\right)
$ \\
\item[-]and whose differential is given by $d_{0}^{W(G)}({\bf f}(C ; \sigma))={\bf z}_{C}$ for $i=0$ and for $i>0$ is defined as follows: 
\begin{align*}
&d_{i}({\bf f}(C ; \sigma)) \\
&=\displaystyle\sum_{x^{\prime}_{s}\in\sigma}(-1)^{\alpha(\sigma;x^{\prime}_{s})}[y^{\prime}_{s}{\bf f}((C\setminus \{y^{\prime}_{s}\})\cup \{x^{\prime}_{s}\};\sigma\setminus \{x^{\prime}_{s}\})-x^{\prime}_{s}{\bf f}(C;\sigma\setminus \{x^{\prime}_{s}\})],
\end{align*}
\end{enumerate}
where $\alpha(\sigma;x^{\prime}_{s})=|\{x^{\prime}_{j}\in\sigma\,\,: j>s\}|$. 
\end{const}

\begin{thm}\label{min resol3}
Let $G=H(n_{1},\ldots, n_{d_{0}})$ be a very well-covered graph. Then the complex $(\mathbb{F},d_{\bullet})$ given in Construction \ref{const3} is a minimal free resolution of $J(G)$
\end{thm}
\begin{proof}
Simply replace $x_{i}$ with $x^{\prime}_{i}$, and $y_{i}$ with $y^{\prime}_{i}$ in the proof of \cite[Theorem 3.5]{cf}. 
\end{proof}

For a subset $B$ of $V(H)$, we set $N_{B}=\displaystyle\sum_{j}n_{j}$, where the sum runs over those integers $j$ such that $x^{\prime}_{j}\in B$ or $y^{\prime}_{j}\in B$. 

\begin{cor}
Let $G=H(n_{1},\ldots, n_{d_{0}})$ be a very well-covered graph. Then we have 
\begin{enumerate}
\item $\beta_{i,i+j}(J(G))=\displaystyle\sum_{C\in{\rm Min}(H)}|\{D\subset\mathcal{C}(H; C)\,\,: d+N_{D}=j \mbox{ and }|D|=i\}|$
\item ${\rm pd}(J(G))=\max\{|\mathcal{C}(H; C)|\,\,:C\in{\rm Min}(H)\}.$
\end{enumerate}
In particular, the graded Betti numbers of $J(G)$ do not depend upon the characteristic of the underlying field $\Bbbk$
\end{cor}

\begin{cor}\label{local cohomology of very well}
Let $G=H(n_{1},\ldots, n_{d_{0}})$ be a very well-covered graph and let $S=\Bbbk[V(G)]$. Then, for a positive integer $j$, the Hilbert series of $H^{j}_{\mathfrak{m}}(S/I(G))$ is given by the following formula: 
\begin{align*}
&F(H^j_{\mathfrak{m}}(S/I(G)),t) \\[0.15cm]
& =\displaystyle\sum_\ell\dim_\Bbbk[H^j_{\mathfrak{m}}(S/I(G))]_{\ell} \ t^{\ell} \\[0.15cm]
& =\displaystyle\sum_{i}\displaystyle\sum_{C\in{\rm Min}(H)}|\{D\subset\mathcal{C}(H; C)\,\,: d+N_{D}=|V(G)|-j \mbox{ and }|D|=i\}|\left(\frac{t^{-1}}{1-t^{-1}}\right)^{j-i}.
\end{align*}
\end{cor}

\begin{rem}
It is known that the Hilbert series of the local cohomology modules of $S/I(G)$ for arbitrary very well-covered graphs in \cite[Theorem 4.4]{kpty}. Although Corollary \ref{local cohomology of very well} is stated in a different form from \cite[Theorem 4.4]{kpty}, they are equivalent. 
\end{rem}


\section*{Acknowledgement}
The authors would like to express their gratitude to Tatsuya Kataoka for his helpful and valuable comments and for carefully reading this paper. 
 The research of Yuji Muta was partially supported by ohmoto-ikueikai and JST SPRING Japan Grant Number JPMJSP2126. 
 The research of Naoki Terai was partially supported by JSPS Grant-in Aid for Scientific Research (C) 24K06670.



\begin{thebibliography}{99}

\bibitem{ahh} A. Aramova, J. Herzog and T. Hibi, 
Squarefree lexsegment ideals, 
\textit{Math. Z.} \textbf{228} (1998), no. 2, 353--378.

\bibitem{bps} D. Bayer, I. Peeva and B. Sturmfels, 
Monomial resolutions, 
\textit{Math. Res. Lett.} \textbf{5} (1998), 31--46.

\bibitem{bs} D. Bayer and B. Sturmfels, 
Cellular resolutions of monomial modules, 
\textit{J. Reine Angew. Math.} \textbf{502} (1998), 123--140.

\bibitem {bw} A. Bj\"{o}rner, M.L. Wachs, 
Shellable nonpure complexes and posets II, 
{\it Trans. Amer. Math. Soc.} {\bf 349} (1997), no. 10, 3945--3975. 

\bibitem{bh} W. Bruns and J. Herzog, 
\textit{Cohen--Macaulay Rings}, 
Cambridge Studies in Advanced Mathematics, vol. 39, Cambridge Univ. Press, Cambridge, 1993.

\bibitem{cw} K. Cameron and T. Walker, 
The graphs with maximum induced matching and maximum matching the same size, 
\textit{Discrete Math.} \textbf{299} (2005), 49--55.

\bibitem{c} R. Chen, 
Minimal free resolutions of linear edge ideals, 
\textit{J. Algebra} \textbf{324} (2010), no. 12, 3591--3613.

\bibitem{cn} D. Cook II and U. Nagel, 
Cohen--Macaulay graphs and face vectors of flag complexes, 
\textit{SIAM J. Discrete Math.} \textbf{26} (2012), 89--101.

\bibitem{cf} M. Crupi and A. Ficarra, 
Very well-covered graphs by Betti splittings, 
\textit{J. Algebra} \textbf{629} (2023), 76--108.

\bibitem{crt} M. Crupi, G. Rinaldo and N. Terai, 
Cohen--Macaulay edge ideal whose height is half of the number of vertices, 
\textit{Nagoya Math. J.} \textbf{201} (2011), 117--131.

\bibitem{ds} H. Dao and J. Schweig, 
Projective dimension, graph domination parameters, and independence complex homology, 
\textit{J. Combin. Theory Ser. A} \textbf{116} (2009), 180--197.

\bibitem{j} J. A. Eagon, 
Partially split double complexes with an associated Wall complex and applications to ideals generated by monomials, 
\textit{J. Algebra} \textbf{135} (1990), no. 2, 344--362.

\bibitem{er} J. A. Eagon and V. Reiner, 
Resolutions of Stanley--Reisner rings and Alexander duality, 
\textit{J. Pure Appl. Algebra} \textbf{130} (1998), 265--275.

\bibitem{ek} S. Eliahou and M. Kervaire, 
Minimal resolutions of some monomial ideals, 
\textit{J. Algebra} \textbf{129} (1990), 1--25.

\bibitem{hh} J. Herzog and T. Hibi, 
Distributive lattices, bipartite graphs and Alexander duality, 
\textit{J. Algebraic Combin.} \textbf{22} (2005), 289--302.

\bibitem{hhz} J. Herzog, T. Hibi and X. Zheng, 
Cohen--Macaulay chordal graphs, 
\textit{J. Combin. Theory Ser. A} \textbf{113} (2006), 911--916.

\bibitem{hhko} T. Hibi, A. Higashitani, K. Kimura and A. B. O'Keefe, 
Algebraic study on Cameron--Walker graphs, 
\textit{J. Algebra} \textbf{422} (2015), 257--269.

\bibitem{hp} D. T. Hoang and M. H. Pham, 
The size of Betti tables of edge ideals of clique corona graphs, 
\textit{Arch. Math. (Basel)} \textbf{118} (2022), 577--586.

\bibitem{h} M. Hochster, 
Cohen--Macaulay rings, combinatorics and simplicial complexes, in: B. R. McDonald and R. A. Morris (eds.), \textit{Ring Theory II}, Lecture Notes in Pure and Appl. Math., vol. 26, Marcel Dekker, New York, 1977, pp. 171--223.

\bibitem{h1} N. Horwitz, 
Linear resolutions of quadratic monomial ideals, 
\textit{J. Algebra} \textbf{318} (2007), no.~2, 981--1001.

\bibitem{k} M. Katzman, 
Characteristic-independence of Betti numbers of graph ideals, 
\textit{J. Combin. Theory Ser. A} \textbf{113} (2006), 435--454.

\bibitem{kty} K. Kimura, N. Terai and S. Yassemi, 
The projective dimension of the edge ideal of a very well-covered graph, 
\textit{Nagoya Math. J.} \textbf{230} (2018), 160--179.

\bibitem{kpty} K. Kimura, M. R. Pournaki, N. Terai and S. Yassemi, 
Very well-covered graphs and local cohomology of their residue rings by the edge ideals, 
\textit{J. Algebra} \textbf{606} (2022), 1--18.

\bibitem{lj} E. Lashani and A. S. Jahan, 
Regularity and projective dimension of some class of well-covered graphs, 
\textit{Turk. J. Math.} \textbf{40} (2016), 1102--1109.

\bibitem{mm} F. Mohammadi and S. Moradi, 
Resolution of unmixed bipartite graphs, 
\textit{Bull. Korean Math. Soc.} \textbf{52} (2015), 977--986.

\bibitem{mpt} Y. Muta, M. R. Pournaki and N. Terai, 
A local cohomological viewpoint on edge rings associated with multi-whisker graphs, 
\textit{Comm. Algebra} \textbf{53} (2025), no. 5, 1856--1865.

\bibitem{s0} R. P. Stanley, 
Hilbert functions of graded algebras, 
\textit{Adv. Math.} \textbf{28} (1978), 57--83.

\bibitem{s1} R. P. Stanley, 
\textit{Combinatorics and Commutative Algebra}, 
2nd ed., Progress in Mathematics, vol.~41, Birkhäuser Boston, Boston, MA, 1996.

\bibitem{t} N. Terai, 
Alexander duality theorem and Stanley--Reisner rings, 
\textit{Sūrikaisekikenkyūsho Kōkyūroku} \textbf{1078} (1999), 174--184.

\bibitem{v0} A. Van Tuyl, 
A beginner's guide to edge and cover ideals, in: 
\textit{Monomial Ideals, Computations and Applications}, Lecture Notes in Math., vol.~2083, Springer, Heidelberg, 2013, pp.~63--94.

\bibitem{v1} R. H. Villarreal, 
Cohen--Macaulay graphs, 
\textit{Manuscripta Math.} \textbf{66} (1990), no. 3, 277--293.

\bibitem{v2} R. H. Villarreal, 
\textit{Monomial Algebras}, 
2nd ed., Monographs and Research Notes in Mathematics, CRC Press, Boca Raton, FL, 2015.

\bibitem{y} S. Yuzvinsky, 
Taylor and minimal resolutions of homogeneous polynomial ideals, 
\textit{Math. Res. Lett.} \textbf{6} (1999), no. 5–6, 779--793.


\end{thebibliography}
\end{document}